\documentclass[12pt,a4paper,oneside]{amsart}

\usepackage{amsfonts, amsmath, amssymb, amsthm, amscd, hyperref}

\usepackage{graphicx}

\usepackage{anysize}
\marginsize{2cm}{2cm}{2cm}{2cm}

\newtheorem{theorem}{Theorem}
\newtheorem{lemma}{Lemma}

\newtheorem{corollary}{Corollary}

\theoremstyle{definition}
\newtheorem{definition}{Definition}
\newtheorem*{definition*}{Definition}

\theoremstyle{remark}
\newtheorem{remark}{Remark}

\DeclareMathOperator{\Pow}{Pow}
\DeclareMathOperator{\Vol}{Vol}
\DeclareMathOperator{\grad}{grad}

\begin{document}

\ifpdf
\DeclareGraphicsExtensions{.pdf, .jpg, .tif, .mps}
\else
\DeclareGraphicsExtensions{.eps, .jpg, .mps}
\fi

\title{Some Remarks on the Circumcenter of Mass}
\author{Arseniy V. Akopyan}

\thanks{Institute for Information Transmission Problems RAS\\ 
			Bolshoy Karetny per. 19, Moscow, Russia 127994  \\        %  \\
%             \emph{Present address:} of F. Author  %  if needed
B.~N.~Delone International Laboratory ``Discrete and Computational Geometry'', P.~G~ Demidov Yaroslavl State University., Sovetskaya st. 14, Yaroslavl', Russia 150000
\\
% \email{akopjan@gmail.com}
\email{akopjan@gmail.com}}

\date{Received: date / Accepted: date}

\maketitle

\begin{abstract}
In this article we give new proofs for the existence and basic properties of the cirucmcenter of mass defined by V.~E.~Adler in \cite{adler1993recuttings} and S.~Tabachnikov and E.~Tsukerman in~\cite{tabachnikov2014circumcenter}.
\end{abstract}

We start with definitions.

\begin{definition}
	\label{def:power of point}
	\emph{The power of a point} $\mathbf{x}$ with respect to a sphere $\omega(\mathbf{o}, R)$ in $\mathbb{R}^d$ is defined $\Pow(\omega, \mathbf{x})=\|\mathbf{ox}\|^2-R^2$. Here $\mathbf{o}$ is the center and $R$ the radius of the sphere $\omega(\mathbf{o}, R)$.
\end{definition}

\begin{definition}
	\label{def:power of simplex}
	Given simplex $\Delta$ in $\mathbb{R}^d$, define
	\[
	\Pow(\Delta)=\int\limits_\Delta \Pow(\omega_\Delta, \mathbf{x})d\mathbf{x},
	\]
	where $\omega_\Delta$ is the circumsphere of $\Delta$.
\end{definition}

\begin{remark}
	If the sphere $\omega$ is a sphere of higher dimension passing through all vertices of $\omega_\Delta$ then the power of any point of $\Delta$ with respect to $\omega$ is the same as the power with respect to $\omega_{\Delta}$.
	Therefore in the definition of $\Pow(\Delta)$ the circumcscribed sphere could be changed to any sphere passing through vertices of $\Delta$.
\medskip

	Note also, that the value of $\Pow(\Delta)$ is always negative.
\end{remark}

% Этот интеграл можно интерпретировать, как объем зажатый в призме построенной над $\Delta$  ограниченной снизу параболоидом $P(\mathbf{x})$, а сверху плоскостью содержащей $P(\omega_\Delta)$.
Denote the vertices of $\Delta$ by $\mathbf{v}_0$, $\mathbf{v}_1$, $\dots$, $\mathbf{v}_d$, let $\mathbf{o}_\Delta$ and $R_\Delta$ be the center and the radius of the circumsphere, and let $\mathbf{m}_\Delta$ be the centroid of $\Delta$.
Then one has the following formulas for $\Pow(\Delta)$ (see.~\cite{Rajan1994Optimality}),

\begin{equation*}
		-\Pow(\Delta)=\frac{\Vol(\Delta)}{(d+1)(d+2)}
		\left(\sum_{i=0}^{d} \sum_{j=0}^{i-1} \|\mathbf{v}_i\mathbf{v}_j\|^2 \right)=\frac{d+1}{d+2}\Vol(\Delta)\left(R_\Delta^2-\|\mathbf{o}_\Delta \mathbf{m}_\Delta\|^2 \right),
\end{equation*}

where $\Vol(\Delta)$ is the volume of $\Delta$.

\begin{lemma}
	\label{lem:base equality}
	Given a simplex $\Delta$ in $\mathbb{R}^d$, denote by $\vec{\mathbf{n}}_i$ the unit normal to the hyperface $\Delta_i$ directed in the exterior of $\Delta$.
	Then
	\[
	\sum\limits_{i=0}^d \Pow(\Delta_i)\vec{\mathbf{n}}_i=
	2\Vol(\Delta) \overrightarrow{\mathbf{o}_\Delta \mathbf{m}_\Delta}.
	\]
\end{lemma}
\begin{proof}
	Without loss of generality, assume that $\mathbf{o}_\Delta$ is the origin.
	
	Let us use the following variant of the Gauss--Ostrogradsky theorem, also known as the gradient theorem:	
	\[
	\int\limits_{\Delta} \grad f(\mathbf{x})dv= \int \limits_{\partial \Delta} f(\mathbf{x}) \vec{\mathbf{n}}(\mathbf{x}) ds, 
	\]
	where $dv$ and $ds$ are the volume elements of total space and of the surface of the simplex, and $\vec{\mathbf{n}}(\mathbf{x})$ is the unit normal to the surface at a point $\mathbf{x}$.

	Apply this equation to the power of a point with respect to the circumsphere, $f(\mathbf{x})=\|\mathbf{x}\|^2-R_\Delta^2$.
	Then $\grad f(\mathbf{x})=2\mathbf{x}$. 
	Note also that $\displaystyle \int_{\Delta_i} f(\mathbf{x})ds=\Pow(\Delta_i)$ since the sphere $(\mathbf{o}, R_\Delta)$ passes through the vertices of $\Delta_i$.
	We obtain
	\[
		2\Vol(\Delta) \overrightarrow{\mathbf{o}_\Delta \mathbf{m}_\Delta}=\int\limits_{\Delta} 2\mathbf{x}dv= \int \limits_{\partial \Delta} f(\mathbf{x})\vec{\mathbf{n}}(\mathbf{x}) ds=\sum\limits_{i=0}^d \Pow(\Delta_i)\vec{\mathbf{n}}_i.
	\]
\end{proof}

\begin{corollary}
	\label{cor:main theorem}
	Let $\mathcal{C}$ be a $d$-dimensional piece-wise linear simplicial cycle in $\mathbb{R}^d$.
	Let $\mathbf{o}_i$ and $\mathbf{m}_i$ be circumcenters and centroids of $d$-dimensional simplices $\Delta_i\in \mathcal{C}$.
	Then 
	\[
	\sum\limits_{\Delta_i\in \mathcal{C}} \overrightarrow{\mathbf{o}_i \mathbf{m}_i}\Vol(\Delta_i)=\mathbf{0}.
	\]
\end{corollary}

For the centroid, one has $\sum\limits_{\Delta_i\in \mathcal{C}} \vec{\mathbf{m}_i}\Vol(\Delta_i)=\mathbf{0}$, because each point is counted the same number of times with positive and negative sign.
So, we obtain the following corollary.

\begin{corollary}[V. E. Adler, S. Tabachnikov, E. Tsukerman]
	\label{cor:theorem for circumcenters}
	Let $\mathcal{C}$ be a $d$-dimensional piece-wise linear simplicial cycle in $\mathbb{R}^d$.
	Suppose $\mathbf{o}_i$ be the circumcenters of $d$-dimensional simplices $\Delta_i\in \mathcal{C}$.
	Then 
	\[
	\sum\limits_{\Delta_i\in \mathcal{C}} \vec{\mathbf{o}}\Vol(\Delta_i)=\mathbf{0}.
	\]
\end{corollary}

Following \cite{tabachnikov2014circumcenter}, we give the following definition.

\begin{definition}
	Let $\mathcal{K}$ be a $d$-dimensional piece-wise linear simplicial chain.
	Let $(\mathbf{o}_i, \Vol(\Delta_i))$ be the weighted point located at the circumcenter of $\Delta_i\in \mathcal{K}$ with the weight $\Vol(\Delta_i)$.
	The center of mass of points $(\mathbf{o}_i, \Vol(\Delta_i))$ of all simplices of $\mathcal{K}$ is called the \emph{circumcenter of mass} of $\mathcal{K}$.
\end{definition}

\begin{remark}
	\label{rem:(d-1)-cycle}
	We can define the circumcenter of mass of any $(d-1)$-dimensional piece-wise linear simplicial cycle $\mathcal{C}$ in $\mathbb{R}^d$ as the circumcenter of mass of any its filling, that is $\mathcal{K}$ such that $\partial \mathcal{K}= \mathcal{C}$. Due to Corollary~\ref{cor:theorem for circumcenters}, the choice of filling for $\mathcal{C}$ does not matter.
\end{remark}
\medskip

It seems that the first who noted the existence of circumcenter of mass of planar polygon was Giusto Bellavitis in 1834 (see the book \cite{laisant1887theorie}, pages 150--151).

In 1993 it was independently noticed by V.~E.~Adler in \cite{adler1993recuttings} for case of triangulation of planar polygon  by diagonals and in the private correspondence of G.C.~Shephard and B.~Gr\"unbaum. They also noted that the circumcenter could be replaced by any point on the Euler line, that is, by a fixed affine combination of the centroid and the circumcenter (for example, the orthocenter or the center of the Euler circle).

A.G~Myakishev in \cite{myakishev2006two} prove the existence of Euler (and also Nagel) line for quadrilateral.

S.~Tabachnikov and E.~Tsukerman in \cite{tabachnikov2014circumcenter} proved the correctness of definition of circumcenter of mass for any simplicial polytope and the existence of the Euler line in a high dimensional polytope.

The case of central triangulation of a tetrahedron was posed on the student contest IMC 2009 (Problem 5). 

In the planar case, we can take a polygon as a cycle. Using Lemma~\ref{lem:base equality} we can give a short proof of the following theorem proved by S.~Tabachnikov and E.~Tsukerman.

\begin{theorem}[\cite{tabachnikov2014circumcenter}]
	\label{thm:equileteral polygon}
	Let $P=\mathbf{a}_1\mathbf{a}_2\dots \mathbf{a}_n$ be an equilateral polygon.
	Then its circumcenter of mass coincide with centroid of the polygonal lamina.
\end{theorem}
\begin{proof}
	Denote by $\mathbf{o}$ and $\mathbf{m}$ the circumcenter of mass and the centroid.
	Note that, for all $i$, the values $\Pow(\mathbf{a}_i\mathbf{a}_{i+1})$ are equal to each other. Denote this quantity by $p$ and let $l$ be the length of the sides.
	We have $\Vol(P)\overrightarrow{\mathbf{m} \mathbf{o}}=\frac12\sum\limits_{i=1}^n p\vec{\mathbf{n}}_i$.
	Note that this sum is equal to zero because each vector $\vec{\mathbf{n}}_i$ is the vector $\frac{1}{l}\overrightarrow {a_ia_{i+1}}$ rotated by $90^\circ$, but $\sum \limits_{i=1}^n{a_ia_{i+1}}=0$.
\end{proof}

Note that using the equation on $\Pow(\Delta)$ we can generalize this theorem to higher dimension.

\begin{theorem}
	\label{thm:equileteral polytope}
	Let $P$ be a simplicial polytope in $\mathbb{R}^d$, such that for each face of $P$ the sum of squares of its edges is a constant.
	Then the circumcenter of mass and centroid of solid polytope $P$ coincide.
\end{theorem}
\begin{proof}
	Denote this constant by $c$. 
	Note that for each facet $\Delta_i$ of $P$ we have:
	\[
	\Pow(\Delta_i)=\Vol(\Delta_i) \frac{-c}{(d+1)(d+2)}.
	\]
	
	From Minkowski's theorem it follows that 
	\[
	\sum_{\Delta_i \in (\text{facets of }P)}\Vol(\Delta_i) \mathbf{n}_i=0,
	\]
	where $\mathbf{n}_i$ is a unit normal to the facet $\Delta_i$.
	
	Therefore
	\begin{equation*}
		\Vol(P)\overrightarrow{\mathbf{m} \mathbf{o}}=
		\frac{1}{2} \sum_{\Delta_i \in (\text{facets of }P)}\Pow(\Delta_i) \mathbf{n}_i % =\\
		=\frac{-c}{2(d+1)(d+2)}\sum_{\Delta_i \in (\text{facets of }P)}\Vol(\Delta_i) \mathbf{n}_i=0.
	\end{equation*}
\end{proof}

\begin{remark}
	Using another formula for $\Pow(\Delta_i)$ we can reformulate the requirements on the facets of the polytope $P$ by the following way:
	for each facet $\Delta_i$ the value $R_{\Delta_i}^2-\|\mathbf{o}_{\Delta_i} \mathbf{m}_{\Delta_i}\|^2$ is a constant.
\end{remark}

As the authors mention in \cite{tabachnikov2014circumcenter}, if the vertices of  $\mathcal{C}$ lie on a sphere $\omega$, then the circumcenter of mass coincide with center of the sphere. Indeed, there is a filling of $\mathcal{C}$ with the same set of vertices as $\mathcal{C}$.  The circumcenters of the simplices of this filling coincide the center of the sphere~$\omega$.

In the same article S.~Tabachnikov and E.~Tsukerman give a definition of cirumcenter of mass in the spherical geometry. Using the previous observation, we can give another explanation of the existence of this point.

Consider the unit sphere $\mathcal{S}^d$ with the center at the origin $\mathbf{o}$ of $\mathbb{R}^{d+1}$.
By a weighted point $(\mathbf{x}, m)$ we mean a pair consisting of a point $\mathbf{x}$ and a number $m$, which is natural to interpret as vector $m\vec{\mathbf{x}}$ in $\mathbb{R}^{d+1}$.
A set of weighted points $(\mathbf{x}_i, m_i)$ has the centroid at point $\frac{\sum{m_i\vec{\mathbf{x}}_i}}{\|\sum{m_i\vec{\mathbf{\mathbf{x}}}_i}\|}$ and the total mass $\|\sum{m_i\vec{\mathbf{x}}_i}\|$ (See~\cite{galperin1993concept}).

For each spherical $d$-simplex $\Delta_i=\mathbf{v}_0\mathbf{v}_1\dots\mathbf{v}_d$ of a spherical simplicial chain $\mathcal{C}$, consider a point $\mathbf{o}'_i$ which is the circumcenter of the simplex $\Delta'_i=\mathbf{o}\mathbf{v}_0\mathbf{v}_1\dots\mathbf{v}_d$ in $\mathbb{R}^{d+1}$.
Now we can define the weighted circumcenter as the point  $\mathbf{o}_i=\left(\displaystyle \frac{\mathbf{o}'_i}{\|\mathbf{o}'_i\|}, \Vol(\Delta'_i)\|\mathbf{o}'_i\|\right)$\footnote{Using simple calculation it is easy to show that $\Vol(\Delta'_i)\|\mathbf{o}'_i\|=\frac{\Vol(\Delta_i)}{2(d+1)}$. So the circumcenter of mass from \cite{tabachnikov2014circumcenter} is the same as here.}. The $(d+1)$-dimensional complex in $\mathbb{R}^{d+1}$ formed by the simplices $\Delta'_i$ is denoted by $\mathcal{C}'$.

\begin{corollary}
	Suppose $\mathcal{C}$ is a $d$-dimensional simplicial cycle $\mathcal{C}$ in $\mathcal{S}^d\subset \mathbb{R}^{d+1}$. Then its spherical circumcenter of mass coincides with $\mathbf{o}$ (has  zero weight).
\end{corollary}

\begin{proof}
	By definition, the spherical circumcenter of mass $\mathcal{C}$ coincides with the Euclidean circumcenter of mass of $\mathcal{C}'$. But its circumcenter of mass coincides with the circumcenter of mass of $\partial \mathcal{C}'$ which is the origin, because $\partial \mathcal{C}'$ is inscribed in $\mathcal{S}^d$.	
\end{proof}

As in Remark~\ref{rem:(d-1)-cycle}, we can define the circumcenter of mass of a $(d-1)$-dimensional spherical simplicial cycle in $\mathcal{S}^d$ as the circumcenter of mass of its filling.

\bigskip
The author thanks Sergei Tabachnikov for useful discussions and valuable advice.
	
% \bibliographystyle{abbrv}
% %\bibliography{}
% 
% 
% \bibliography{../BibTeX/Bibliography}{}

\end{document}